\documentclass[12pt]{amsart}
\usepackage{amssymb,amsmath, amsthm,latexsym,verbatim}
\usepackage{a4wide}
\usepackage[hypertex]{hyperref}\usepackage{enumitem}

\theoremstyle{plain}

\newtheorem{lem}{Lemma}[section]
\newtheorem{theo}[lem]{Theorem}
\newtheorem{prop}[lem]{Proposition}

\parskip=\medskipamount
\font\k=cmr8
\font\rm=cmr12
  \renewcommand {\ln}{\log}
  \newcommand {\rp}{\rho}
  \newcommand {\rep}{\chi}

  \newcommand {\pr}{\mbox{\k prime}}

  \newcommand {\free}{\mbox{\k free}}
  \newcommand {\tors}{\mbox{\k tors}}
  \newcommand {\C}{{\mathbb C}}
  \newcommand {\bH}{{\mathbb H}}
  \newcommand {\N}{{\mathbb N}}
  \newcommand {\R}{{\mathbb R}}
  \newcommand {\Z}{{\mathbb Z}}
  \newcommand {\Q}{{\mathbb Q}}

  \newcommand {\F}{{\mathbb F}}

  \newcommand {\ho}{{\mathfrak o}}

  \newcommand {\pg}{{\mathfrak p}}

  \newcommand {\cO}{{\mathcal O}}

\newcommand {\cV}{{\mathcal V}}

\newcommand {\cM}{{\mathcal M}}

 \newcommand {\cH}{{\mathcal H}}

\newcommand {\cT}{{\mathcal T}}

\newcommand{\cX}{{\mathcal X}}
\newcommand {\bs}{\backslash}
\newcommand {\ol}{\overline}

\newcommand {\bG}{{\bf G}}

\renewcommand{\Re}{\operatorname{Re}}
\newcommand{\Tr}{\operatorname{Tr}}
\newcommand{\Spec}{\operatorname{Spec}}
\newcommand{\End}{\operatorname{End}}

\newcommand{\I}{\operatorname{I}}
\newcommand{\rk}{\operatorname{rank}}
\newcommand{\vol}{\operatorname{vol}}

\newcommand{\SL}{\operatorname{SL}}
\newcommand{\GL}{\operatorname{GL}}

\newcommand{\SU}{\operatorname{SU}}

\newcommand{\PGL}{\operatorname{PGL}}
\newcommand{\Gal}{\operatorname{Gal}}
\newcommand{\Aut}{\operatorname{Aut}}

\renewcommand{\det}{\operatorname{det}}
\newcommand{\Sym}{\operatorname{Sym}}
\newcommand{\Int}{\operatorname{Int}}
\newcommand{\ord}{\operatorname{ord}}
\newcommand{\ra}{\operatorname{rk}}

\begin{document}
\title[]
{On  the torsion in the cohomology of arithmetic hyperbolic 3-manifolds}
\date{\today}

\author{Simon Marshall}
\address{The Institute for Advanced Study\\
Einstein Drive\\
Princeton\\
NJ 08540, USA}
\email{slm@math.princeton.edu}

\author{Werner M\"uller}
\address{Universit\"at Bonn\\
Mathematisches Institut\\
Beringstrasse 1\\
D -- 53115 Bonn, Germany}
\email{mueller@math.uni-bonn.de}

\begin{abstract}
In this paper we consider the cohomology of a closed arithmetic hyperbolic
$3$-manifold with coefficients in the local system defined by the even 
symmetric powers of the standard representation of $\SL(2,\C)$. The cohomology
is defined over the integers and is a finite abelian group. 
We show that the order of
the 2nd cohomology grows exponentially as the local system grows. We also 
consider the twisted Ruelle zeta function of a closed arithmetic hyperbolic 
$3$-manifold and we express the leading coefficient of its Laurent expansion 
at the origin in terms of the orders of the torsion subgroups of the cohomology.
\end{abstract}

\maketitle
\setcounter{tocdepth}{1}
\section{Introduction}

Let $\bG$ be a semi-simple connected algebraic group over $\Q$, $K$ a 
maximal compact  subgroup of its group of real points $G=\bG(\R)$ and $S=
G/K$ the associated Riemannian symmetric space. Let $\Gamma\subset \bG(\Q)$
be an arithmetic subgroup and $X=\Gamma\bs S$ the corresponding locally
symmetric space. Let $\rho\colon \bG\to \GL(V)$ be a rational representation 
of $\bG$ on a $\Q$-vector space $V$. Then there exists a lattice $M\subset V$
which is stable under $\Gamma$. Let $\cM$ be the associated local system 
of free $\Z$-modules over $X$ defined by the $\Gamma$-module $M$ and let
$H^*(X,\cM)$ be the corresponding sheaf cohomology. 
Since $X$ has the homotopy type of a finite
$CW$-complex, $H^*(X,\cM)$ are finitely generated abelian groups. The 
cohomology of arithmetic groups has important connections to the theory of 
automorphic forms and number theory \cite{Sch}. In this respect,
$H^*(X,\cM\otimes\C)$ has been studied to a great extent. Much less is known 
about the torsion subgroup $H^*(X,\cM)_{\tors}$. In a recent paper, \cite{BV}
Bergeron and Venkatesh  studied the growth of $|H^j(\Gamma_N\bs S,\cM)_{\tors}|$
as $N\to \infty$, where $\{\Gamma_N\}$ is a decreasing sequence of normal 
subgroups of finite index of $\Gamma$ with trivial intersection. This is 
motivated by conjectures that
torsion classes in the cohomology of arithmetic groups should have arithmetic
significance \cite{AS}, \cite{ADP}.  In this paper we consider a similar 
problem in a different aspect. We fix the discrete group and vary the local 
system. We also restrict
attention to the case of hyperbolic 3-manifolds. However, we expect that the
results hold in greater generality.

 Now we explain our results in more detail.
Let $\bH^3=\SL(2,\C)/\SU(2)$ be the $3$-dimensional hyperbolic space and
let $\Gamma\subset \SL(2,\C)$ be a discrete torsion free co-compact subgroup.
Then $X=\Gamma\bs\bH^3$ is a compact oriented hyperbolic $3$-manifold.
We are interested in arithmetic subgroups $\Gamma$ which are  derived from a 
quaternion division algebra $D$ over an imaginary quadratic number field $F$.
The division algebra $D$ determines an algebraic group $\SL_1(D)$ over $F$
which is an inner form of $\SL_2/F$. Moreover, the group of its $F$-rational
points $\SL_1(D)(F)$ is equal to 
\[
D^1:=\{x\in D\colon N(x)=1\},
\]
where $N(x)=x\overline x$ denotes the norm of $x\in D$. Let $\ho\subset D$ be 
an order in $D$ and let $\ho^1=\ho\cap D^1$. Then $\ho^1$ is a discrete subgroup
of $D^1=\SL_1(D)(F)$. The quaternion division algebra $D$ splits over $\C$, 
i.e., there is an isomorphism $\varphi\colon D\otimes_F \C\to M_2(\C)$ of
$\C$-algebras. Let $\Gamma:=\varphi(\ho^1)$. Then $\Gamma$ is a cocompact
arithmetic subgroup of $\SL(2,\C)$. For $n\in\N$ let $V(n)=S^n(F^2)$ be the
$n$-th symmetric power of $F^2$ and let $\Sym^n$ be the $n$-the symmetric
power representation of $\SL_2/F$ on $V(n)$. It follows from Galois descent
that for every even $n$ there is a rational representation  
\[
\mu_n\colon \SL_1(D)/F\to \GL(V(n))
\]
which is equivalent to $\Sym^n$ over $\overline F$. Using this 
representation it follows that for each even $n$ there is a lattice 
$M_n\subset V(n)$ which is stable under $\Gamma$ with respect to $\Sym^n$.  
Let $\cM_n$ be the associated local system of free $\Z$-modules over $X$. Then 
$H^*(X,\cM_n\otimes\R)=0$. Thus $H^*(X,\cM_n)$ is a torsion group. Let
$|H^p(X,\cM_n)|$ denote the order of $H^p(X,\cM_n)$. The purpose of this
paper is to study the behavior of $\ln |H^p(X,\cM_n)|$ as $n\to\infty$.
Our main result is the following theorem.
\begin{theo}\label{thm-asymp}
For every choice of $\Gamma$-stable lattices $M_{2k}$ in $S^{2k}(\C^2)$
we have
\begin{equation}\label{asymp1}
\lim_{k\to\infty}\frac{\ln|H^2(X,\cM_{2k})|}{k^2}=\frac{2}{\pi}\vol(X).
\end{equation}
 Furthermore, for $p=1,3$ we have
\begin{equation}\label{estim1}
\ln |H^p(X,\cM_{2k})|\ll k\ln k
\end{equation}
uniformly over all choices of lattices $M_{2k}$.
\end{theo}
Note that the left hand side of \eqref{asymp1} is a pure combinatorial 
invariant of $X$. So at 
first sight it looks surprising that the volume appears on the right hand side.
However, this is no contractiction, since we know by the Mostow-Prasad rigidity
theorem that the volume of a hyperbolic manifold is a topological invariant.

The proof of \eqref{asymp1} is a consequence of the following theorem combined 
with \eqref{estim1}.
\begin{theo}\label{thm-asymp2} 
The alternating sum of $\ln |H^p(X,\cM_{2k})|$ is independent 
of the choice of a $\Gamma$-stable lattice $M_{2k}$ in $S^{2k}(\C^2)$
and we have
\begin{equation}\label{asymp2}
\sum_{p=1}^3(-1)^p\ln |H^p(X,\cM_{2k})|=\frac{2}{\pi}\vol(X)k^2+O(k)
\end{equation}
as $k\to\infty$. 
\end{theo}
The fact that $\sum_{p=1}^3(-1)^p\log|H^p(X,\cM_{2k})|$ is independent of the
choice of the $\Gamma$-stable lattice $M_{2k}$ follows from the proof 
of \eqref{asymp2}, but it can also be seen in more elementary way. See the 
remark following \eqref{reidem2}.

More generally for $m,n\in\N$ even, we may consider the local system 
\[
 \cM_{n,m}=\cM_n\otimes \overline \cM_m.
\]
where $\overline\cM_n$ is the local system attached to the complex conjugated 
lattice $\overline M_n$ of $M_n$. If $n\neq m$, then nothing changes. We still 
have $H^*(X,\cM_{n,m}\otimes \R)=0$. Thus for $n\neq m$,
$H^*(X,\cM_{n,m})$ is a torsion group and there is an asymptotic formula similar
to \eqref{asymp1} as $m\to\infty$ or $n\to\infty$. 

In \cite{BV} Bergeron and Venkatesh established results of similar nature 
but in a different aspect. They study the growth of the torsion in
the cohomology for a fixed local system as the lattice varies in a decreasing
sequence of congruence subgroups. Again the volume of the locally symmetric
space appears as the main ingredient of the asymptotic formulas.   

Our next result is related to the Ruelle zeta function of $X$. In \cite{De} 
Deninger discussed a geometric analogue of Lichtenbaum's conjectures in the 
context of Ruelle zeta functions attached to certain dynamical systems. We 
establish a result of similar nature for the Ruelle zeta function of a compact
 arithmetic hyperbolic $3$-manifold. 

The Ruelle zeta function is a 
dynamical zeta function attached to the geodesic flow on the unit 
tangent bundle  of $X$.  We recall its definition.
 Let $\rep\colon \Gamma\to \GL(V)$ be a representation on a 
finite-dimensional complex vector space. Given $\gamma\in\Gamma$, denote by 
$[\gamma]$ the $\Gamma$-conjugacy class of $\gamma$. For 
$\gamma\in \Gamma\setminus\{e\}$ let $\ell(\gamma)$ be the length of the 
unique closed geodesic that corresponds to $[\gamma]$. Then the twisted
Ruelle zeta function is defined as 
\begin{equation}\label{ruelle1}
R(s;\rep):=\prod_{\substack{[\gamma]\neq e\\\pr}}\det\left(\I-\rep(\gamma)
e^{-s\ell(\gamma)}\right)^{-1}.
\end{equation}
The product runs over all nontrivial primitive conjugacy classes. The infinite
product converges in some half-plane $\Re(s)>c$ and  admits a meromorphic 
extension to $\C$ \cite[Sect. 7]{Fr2}. We note that the definition 
\eqref{ruelle1} differs from the usual one by the exponent $-1$. 

We consider the special 
case where $\rep$ is the restriction to $\Gamma$ of a finite-dimensional
representation $\rp$ of $\SL(2,\C)$. We denote the associated twisted 
Ruelle zeta function by $R(s;\rp)$. Note that the irreducible 
finite-dimensional representations of $\SL(2,\C)$, regarded as real Lie group,
are given by 
\begin{equation}\label{irr-rep}
\rp_{m,n}:=\Sym^m\otimes\overline{\Sym^n}, \quad m,n\in\N,
\end{equation}
\cite[p. 32]{Kn}.
Our  interest is in the behavior of $R(s;\rp)$ at $s=0$. To describe it 
we need to introduce the regulator associated to the free part of the 
cohomology of a local system of free $\Z$-modules.

Let $\rp\colon \SL(2,\C)\to \GL(V)$ be a finite-dimensional real representation
of $\SL(2,\C)$ and assume that there exists a lattice $M\subset V$ which is 
stable under $\Gamma$. Let $\cM$ be the associated local system. Let $E$ be
the flat vector bundle over $X$ attached to the restriction of $\rp$ to 
$\Gamma$. By \cite[Lemma 3.1]{MM}, the bundle $E$ can be equipped with a
canonical fibre metric. Let $\cH^*(X;E)$ be the space of $E$-valued harmonic
forms on $X$ with respect to this metric in $E$ and the hyperbolic
metric on $X$. There is a canonical isomorphism
\[
\cH^*(X;E)\cong H^*(X,\cM\otimes\R).
\]
We equip $H^*(X;\cM\otimes\R)$ with the inner product $\langle\cdot,\cdot
\rangle$ induced by the $L^2$-metric on $\cH^*(X;E)$. Let $H^*(X;\cM)_{\free}=
H^*(X;\cM)/H^*(X;\cM)_{\tors}$.  We identify $H^*(X;\cM)_{\free}$ with a 
subgroup of $H^*(X;\cM\otimes\R)$. For each $p=0,1,2,3$ choose a basis
$a_1,...,a_{r_p}$ of $H^*(X;\cM)_{\free}$ and let $G_p(\cM)$ be the Gram matrix
of this basis with respect to the inner product $\langle\cdot,\cdot\rangle$. 
Put $R_p(\cM):=\sqrt{|\det G_p(\cM)|}$. Then the regulator is defined as
\[
R(\cM):=\prod_{p=0}^3 R_p(\cM)^{(-1)^p}.
\]
We can now state our result which describes the behavior of the Ruelle zeta
function at the origin.
 
\begin{theo}\label{value1}
Let $X=\Gamma\bs\bH^3$ be a  compact oriented arithmetic hyperbolic 
$3$-manifold. 
Let $\rp\colon \SL(2,\C)\to\GL(V)$ be an irreducible finite-dimensional 
representation of $\SL(2,\C)$. Let $M\subset V$ be a lattice which is stable
under $\Gamma$ and denote by $\cM$ the associated local system of free
$\Z$-modules over $X$. Let $R(s;\rho)$ be the twisted Ruelle zeta function.
Then we have

\smallskip\noindent
${\mathrm 1)}$ If $\rp\neq 1$, then the order of $R(s;\rp)$ at $s=0$ is 
given by
\begin{equation}\label{order1}
\ord_{s=0}R(s;\rp)=\sum_{q=1}^3(-1)^q q\ra H^q(X;\cM).
\end{equation}
If $\rp=1$, then the order equals $2\dim H^1(X,\R)-4$.

\smallskip\noindent
${\mathrm 2)}$  Let $R^*(0;\rp)$  be the leading coefficient of the Laurent  
expansion of $R(s;\rp)$ at $s=0$. We have 
\begin{equation}\label{value2}
|R^*(0;\rp)|=
R(\cM)^{-1}\cdot\prod_{q=0}^3|H^q(X;\cM)_{\tors}|^{(-1)^q}.
\end{equation}
Moreover, if $\rp=\rp_{m,m}$ for some $m\in\N_0$, then \eqref{value2} holds
for $R^*(0;\rp)$.
\end{theo}
Recall that the irreducible finite-dimensional representations of $\SL(2,\C)$
are given by \eqref{irr-rep}. For each even $n$ there is a $\Gamma$-stable 
lattice in $S^n(\C^2)$. Therefore, if $\rp=\rp_{m,n}$ with $m$ and $n$ even, 
there exist $\Gamma$-stable lattices in the space of the representation. 

We note that there is a formal analogy of \eqref{order1} and \eqref{value2} 
to Lichtenbaum's conjectures on special 
values of Hasse-Weil zeta functions of algebraic varieties 
\cite{Li1}, \cite{Li2}. To make this statement more transparent, we 
recall some details. Let $\cX$ be a regular scheme  which
is separated and of finite type over $\Spec(\Z)$. Let $\zeta_{\cX}(s)$ be the
Hasse-Weil zeta function of $\cX$. It is given by an Euler product that 
converges in some half-plane $\Re(s)>c$.  The Euler
product is expected to have a meromorphic extension to the whole complex 
plane. This is known in some cases.
Lichtenbaum's conjectures are concerned with the behavior of $\zeta_{\cX}(s)$ 
at 
$s=0$. First of all, Lichtenbaum conjectures the existence of a certain 
new cohomology theory for schemes over $\Z$, called ``Weil-\'etale'' 
cohomology. Let  $H^p_c(\cX,\Z)$ be the $p$-th ``Weil-\'etale'' cohomology 
group 
of $\cX$ with compact supports and coefficients in $\Z$. Then the 
conjectures of Lichtenbaum are the following statements: 1) The groups 
$H^p_c(\cX,\Z)$ are finitely
generated and vanish for $p>2\dim\cX+1$. 2) The order of $\zeta_{\cX}(s)$ at 
$s=0$ is given by
\[
\ord_{s=0}\zeta_X(s)=\sum_{p}(-1)^p p\ra H^p_c(\cX,\Z).
\]
3) The leading coefficient $\zeta_{\cX}^*(0)$of the Laurent expansion of 
$\zeta_{\cX}(s)$ 
satisfies
\[
\zeta_{\cX}^*(0)=R^{-1}\cdot \prod_{p}|H^p_c(\cX,\Z)_{\tors}|^{(-1)^{p}},
\]
where $R$ is the Reidemeister torsion of some acyclic complex associated to
$H^*_c(\cX,\R)$, equipped with volume forms determined  by the isomorphism
$H_c^*(\cX,\R)=H^*_c(\cX,\Z)\otimes\R$ and a basis of $H_c^*(\cX,\Z)_{\free}$. 
We note that Denniger \cite{Ch} first discussed a geometric analogue of 
the Lichtenbaum conjectures in the  context of dynamical systems.

Our approach to prove our main results is based
on the study of the Reidemeister torsion of $X$. Let $\rp_{2k}$ be the 
representation of $\Gamma$ obtained by the restriction of $\Sym^{2k}\colon
\SL(2,\C)\to \GL(S^{2k}(\C^2))$ to $\Gamma$. Let $\rp_{2k}^\R$ be the underlying
real representation. Denote by $\tau_X(\rp_{2k}^\R)$ be the
Reidemeister torsion of $X$ with respect to $\rp_{2k}^\R$ (see section 
\ref{reidcohom} for its definition). Then the Reidemeister
torsion satisfies
\[
\tau_X(\rp_{2k}^\R)=\prod_{p=1}^3|H^p(X,\cM_{2k})|^{(-1)^{p+1}}.
\]
This equality was first noted by Cheeger \cite[(1.4)]{Ch}. Now we apply
\cite[Corollary 1.2]{Mu2} which describes the asymptotic behavior of 
$\ln\tau_X(\rp_{2k}^\R)$ as $k\to\infty$. This implies 
Theorem \ref{thm-asymp2}. To estimate $\ln |H^p(X,\cM_{2k})|$ for $p=1,3$, we
use that $H^3(X,\cM_{2k})$ is isomorphic to the space  $(M_{2k})_\Gamma$
of coinvariants. To bound $(M_{2k})_\Gamma$ we can work locally. This
leads to \eqref{estim1}. Together with Theorem \ref{thm-asymp2} we obtain
\eqref{asymp1}, which proves Theorem \ref{thm-asymp}.

To prove Theorem \ref{value1}, we use \cite[Theorem 1.5]{Mu2}. Using this 
theorem we obtain the statement about the order of $R(s;\rp)$ at $s=0$.
Moreover, it follows that the leading coefficient of the Laurent expansion of 
$R(s;\rp)$ at $s=0$ equals $T_X(\rp;h)^{-2}$, where $T_X(\rp;h)$ is the 
Ray-Singer analytic torsion of $X$ and $\rp|_\Gamma$ with respect to the 
canonical fibre metric in the flat bundle $E$ mentioned above. By 
\cite[Theorem 1]{Mu1}, the analytic torsion equals the Reidemeister torsion
$\tau_X(\rp;h)$. If there exists a $\Gamma$-stable lattice in $V$, then the
Reidemeister torsion satisfies 
\begin{equation}\label{reidcohom5}
\tau_X(\rp;h)=R(\cM)^{-1}\prod_{p=0}^3|H^p(X,\cM)_{\tors}|^{(-1)^{p+1}}.
\end{equation} 
Combining these facts, we obtain Theorem \ref{value1}.

The paper is organized as follows. In section \ref{reidcohom} we discuss the
relation between Reidemeister torsion and cohomology, if the chain complex is
defined over the integers. In particular, we prove \eqref{reidcohom5}. In
section \ref{arithmgr} we collect a number of facts about cocompact arithmetic
subgroups of $\SL(2,\C)$ which are derived from quaternion division algebras. 
In particular, we prove that the even symmetric powers contain $\Gamma$-stable
lattices. In section \ref{bounds} we prove \eqref{estim1}.
In the final section \ref{main} we prove our main results.

\noindent
{\bf Acknowledgment.} The second named author would like to thank 
G\"unter Harder and Skip Garibaldi
for some very helpful discussions and suggestions. Especially, the proof of
Lemma 3.1 was suggested by Skip Garibaldi.  The first named author would like 
to acknowledge the generous support of the Institute for Advanced Study.

\section{Reidemeister torsion and cohomology}\label{reidcohom}

We recall some facts about Reidemeister torsion. Let $V$ be a 
finite-dimensional real vector space of dimension $n$. Set $\det V=
\Lambda^n(V)$. A volume element in $V$ is a nonzero element $\omega\in\det V$. 
Any volume element determines an isomorphism $\det V\cong \R$. Furthermore,
note that a volume element is given by $\omega=e_1\wedge e_2\wedge \cdots
\wedge e_n$ for some basis $e_1,...,e_n$ of $V$. 

Let
\begin{equation}\label{complex1}
C^*\colon 0\to C^0\stackrel{d_0}\longrightarrow C^1\stackrel{d_1}
\longrightarrow \cdots \stackrel{d_{n-1}}\longrightarrow C^n \to 0
\end{equation} 
be a complex of finite dimensional $\R$-vector spaces. Let 
\[
B^j:=d_{j-1}(C^{j-1}),\quad Z^j:=\ker(d_j)
\]
and denote by $H^j(C^*):=Z^j/B^j$ the $j$-th cohomology group of 
$C^*$. Assume that for each $j$ we are given volume elements 
$\omega_j\in\det C^j$ and $\mu_j\in \det H^j(C^*)$. Let $\omega=(\omega_0,...,
\omega_n)$ and $\mu=(\mu_0,...,\mu_n)$. Then the Reidemeister torsion 
$\tau(C^*,\omega,\mu)\in\R^+$ of the complex $C^*$ is defined as a certain 
ratio of volumes (see \cite{Mi}). In \cite{RS} Ray and Singer
gave an equivalent definition in terms of the combinatorial Laplacian which
we recall next. In each $C^j$ we choose
an inner product $\langle\cdot,\cdot\rangle_j$ with volume element $\omega_j$. 
Let
\[
d_{j+1}^*\colon C^{j+1}\to C^j
\]
be the adjoint operator to $d_j$ with respect to the inner products in $C^{j}$
and $C^{j+1}$, respectively. Define the combinatorial Laplacian by
\begin{equation}
\Delta_j^{(c)}=d_{j+1}^*d_j+d_{j-1}d_j^*.
\end{equation}
Then $\Delta_j^{(c)}$ is a symmetric non-negative operator in $C^j$. We have the
combinatorial Hodge decomposition
\begin{equation}\label{hodge1}
C^j=\ker(\Delta_j^{(c)})\oplus d_{j-1}(C^{j-1})\oplus d_{j+1}^*(C^{j+1}).
\end{equation}
It implies that there is a canonical isomorphism
\begin{equation}\label{hodge2}
\ker(\Delta^{(c)}_j)\stackrel{\cong}\longrightarrow H^j(C^*).
\end{equation}
The inner product in $C^j$ restricts to an inner product in 
$\ker(\Delta_j^{(c)})$. Using the isomorphism  \eqref{hodge2} we get an inner 
product $\langle\cdot,\cdot \rangle$ in $H^j(C^*)$. Let $h_1,...,h_{d_j}\in
H^j(C^*)$ be a basis such that $\mu_j=h_1\wedge\cdots\wedge h_{d_j}$. Let $G_j$
be the Gram matrix with entries $\langle h_k,h_l\rangle_j$, $1\le k,l\le d_j$.
Put
\[
V(\mu_j)=\sqrt{|\det G_j|}.
\]
Denote by $\det^\prime\Delta^{(c)}_j$ the product of the nonzero eigenvalues 
of 
$\Delta^{(c)}_j$. Then a slight generalization of \cite[Proposition 1.7]{RS}
gives
\begin{lem}\label{laplace1}
We have
\[
\tau(C^*,\omega,\mu)=\prod_{j=0}^n V(\mu_j)^{(-1)^j}\cdot \prod_{j=1}^n
\left(\det^\prime \Delta^{(c)}_j\right)^{(-1)^{j+1}j/2}.
\]
\end{lem}
Now let 
\begin{equation}\label{complex2}
A^*\colon 0\longrightarrow A^0\stackrel{d_0}\longrightarrow A^1\stackrel{d_1}
\longrightarrow \cdots \stackrel{d_{n-1}}\longrightarrow A^n \longrightarrow 0
\end{equation} 
be a complex of free finite rank $\Z$-modules and set $C^j:=A^j\otimes \R$.  
Then we get a complex \eqref{complex1} of $\R$-vector spaces such that each
$C^j$ has a preferred equivalence class of bases coming from bases of $A^j$.
For any two such bases the matrix of change from one to the other is an
invertible matrix with integral entries. Thus $C^j$ is endowed with a 
canonical volume element $\omega_j$. Furthermore, with respect to any of
the bases coming from $A^j$, the coboundary operator
$d_j$ is represented by a matrix with integral entries. Let $H^j(A^*)$  denote
the $j$-th cohomology group of the complex \eqref{complex2}. Let 
$H^j(A^*)_{\tors}$ be the torsion subgroup of $H^j(A^*)$ and let 
$H^j(A^*)_{\free}=H^j(A^*)/H^j(A^*)_{\tors}$. We identify $H^j(A^*)_{\free}$ with
a subgroup of $H^j(C^*)$. Let $\mu_j$ be a volume element of $H^j(C^*)$. Put
\begin{equation}
R(\mu_j):=\vol(H^j(C^*)/H^j(A^*)_{\free}),
\end{equation}
where the volume is computed with respect to $\mu_j$. Denote by $\tau(C^*,\mu)$
the Reidemeister torsion of $C^*=A^*\otimes \R$ with respect to the canonical
volume element $\omega$, which is determined by $A^*$, and the volume the
volume element $\mu\in\prod_{j=0}^n\det H^j(C^*)$. Finally, denote by 
$|H^j(A^*)_{\tors}|$ the order of the finite group $H^j(A^*)_{\tors}$. Then the
following elementary lemma which describes the relation between Reidemeister 
torsion and the torsion of the cohomology of the complex $A^*$ is proved in  
\cite[section 2]{BV}.
\begin{lem}\label{lem-cohtor}
We have
\begin{equation}\label{cohtor}
\tau(C^*,\mu)=\prod_{j=0}^n R(\mu_j)^{(-1)^j}\cdot \prod_{j=0}^n 
|H^j(A^*)_{\tors}|^{(-1)^{j+1}}.
\end{equation}
\end{lem} 
If $\mu_j$ is the volume element associated to 
the equivalence class of bases of $H^j(C^*)$ coming from bases of the 
lattice $H^j(A^*)_{free}$, we have $R(\mu_j)=1$. In this case \eqref{cohtor}
was first stated in \cite[(1.4)]{Ch} and proved for the case where $A^*$ is
acyclic. 

We now turn to the geometric situation. 
Let $X$ be compact $n$-dimensional Riemannian manifold.  
Choose a base point $x_0\in X$. Let
$\Gamma:=\pi_1(X,x_0)$ be the fundamental group of $X$ with respect to $x_0$,
acting on the universal
covering $\tilde X$ of $X$ as deck transformations. Let $V$ be a 
finite-dimensional real vector space and let
\[
\rep\colon \Gamma\to\GL(V)
\]
be a representation of $\Gamma$ on $V$. It defines a flat vector bundle $E$
over $X$.

Fix a smooth triangulation $K$  of $X$. Let $\tilde K$ be the lift of $K$ to 
a smooth triangulation of $\tilde X$. We think of $K$ as being embedded as a
fundamental domain in $\tilde K$, so that $\tilde K$ is the union of the 
translates of $K$ under $\Gamma$. Let $C^q(\tilde K;\Z)$ be the  cochain
group generated by the $q$-simplexes of $\tilde K$. Then $C^q(\tilde K;\Z)$
is a $\Z[\Gamma]$-module. Define the twisted cochain group $C^q(K;V)$ by
\[
C^q(K;V):=C^q(\tilde K;\Z)\otimes_{\Z[\Gamma]}V.
\]
From the cochain complex $C^*(\tilde K,\Z)$ we get the twisted cochain complex
\begin{equation}
C^*(K;V)\colon  0\to C^0(K;V)\stackrel{d_0}\longrightarrow 
C^1(K;V)\stackrel{d_1}\longrightarrow\cdots\stackrel{d_{n-1}}
\longrightarrow C^n(K;V)\to 0.
\end{equation}
Let $H^*(K;V)$ be the cohomology groups. They are independent of $K$ and will be
denoted by $H^*(X;V)$. The cohomology can also be computed by a different
complex.  Let $C^q(K;E)$ be the set
of $E$-valued $q$-cochains \cite[Chapt. VI]{Wh}. An element of $C^q(K;E)$ 
is a function that assigns to each $q$-simplex a section of $E$ on that 
simplex. The corresponding cochain complex of finite-dimensional $\R$-vector 
spaces computes the cohomology groups $H^*(K;E)=H^*(X;E)$. 
By \cite[p. 278]{Wh} there is canonical isomorphisms
\[
H^*(X;V)\cong H^*(X;E).
\]
Assume that a volume element $\theta\in\det V$ is given.
Let $\sigma_j^q$, $j=1,...,r_q$, be the oriented $q$-simplexes of $K$
considered as preferred bases of the $\Z[\Gamma]$-module $C^q(\tilde K;\Z)$.
Let $e_1,...,e_m$ be a basis
of $V$ such that $\theta=\pm e_1\wedge\cdots\wedge e_n$. Then
$ \{\sigma_j^q\otimes e_k\colon j=1,...,r_q,\; k=1,...,m\}$ is a preferred 
 basis of  $C^q(K;V)$. It defines a volume element $\omega_q\in
\det C^q(K;V)$. The volume element depends on several choices 
\cite[p. 727]{Mu2}. Now assume that $\rep$ is  unimodular, i.e., we have 
\[
|\det\rep(\gamma)|=1,\quad \forall\gamma\in\Gamma.
\]
Then $\omega$ is unique up to sign. Let $\mu\in\det H^*(X;V)$ be a volume 
element. Then we can define the Reidemeister torsion 
$\tau(C^*(K;V);\omega,\mu)$. It still depends on $\theta\in\det V$. However,
if the Euler characteristic of $X$ vanishes, then it is also independent of 
$\theta$ \cite[p. 727]{Mu1}.
It is well known that $\tau(C^*(K;V);\omega,\mu)$ is invariant under 
subdivision \cite{Wh}. Since any two smooth triangulations of $X$ have a 
common subdivision, $\tau(C^*(K;V);\omega,\mu)$ is independent of $K$. 
Therefore we may put
\[
\tau_X(\rep;\mu):=\tau(C^*(K;V);\omega,\mu).
\]
Now pick  a  fibre metric $h$ on $E$. Let $\cH^*(X;E)$ be the space of 
$E$-valued harmonic forms on $X$. Then we have the Hodge-de Rham isomorphism
\[
\cH^*(X;E) \cong H^*(X;E)\cong H^*(X;V).
\]
Using this isomorphism, we get an inner product in $H^*(X;V)$
and a corresponding volume element $\mu_h\in\det H^*(X;V)$. Let 
\[
\tau_X(\rep;h):=\tau_X(\rep,\mu_h).
\]

Now assume that there exists a lattice $M\subset V$ which is invariant under
$\Gamma$, i.e., $M$ is a free abelian subgroup of $V$ such that $V=M\otimes\R$ 
and $M$ is invariant under $\Gamma$. Thus $M$
is a finitely generated $\Z[\Gamma]$-module. It defines a local system 
$\cM$ of free $\Z$-modules on $X$. Set
\begin{equation}
C^q(K;M)=C^q(\tilde K;\Z)\otimes_{\Z[\Gamma]}M,\quad q=0,...,n,
\end{equation} 
and let $H^*(K;M)=H^*(X;M)$ denote the cohomology of the corresponding complex.
Now $C^*(K;M)$ is a complex of finitely generated free $\Z$-modules and we have
\[
C^*(K;V)=C^*(K;M)\otimes \R.
\]
So Lemma \ref{lem-cohtor} applies to this situation.
As above, we may also consider the set $C^q(K,\cM)$ of $\cM$-valued 
$q$-cochains \cite[Chapt. VI]{Wh}. The complex $C^*(K;\cM)$ computes the
cohomology $H^*(K;\cM)=H^*(X;\cM)$ and we have a canonical isomorphism
\[
H^*(X,M)\cong H^*(X,\cM).
\]
Each $H^q(X,\cM)$ is
a finitely generated $\Z$-module. Let $H^q(X,\cM)_{\tors}$ be the torsion 
subgroup and
\[
H^q(X;\cM)_{\free}=H^q(X,\cM)/H^q(X,\cM)_{\tors}.
\]
We identify  $H^q(X,\cM)_{\free}$ with a subgroup of $H^q(X,E)$.
Let $\langle\cdot,\cdot\rangle_q$ be the inner product in $H^q(X,E)$ 
induced by the $L^2$-metric on $\cH^q(X,E)$. 
 Let $e_1,...,e_{r_q}$ be a basis of $H^q(X,\cM)_{\free}$ and let $G_q$ be the
Gram matrix with entries $\langle e_k,e_l\rangle$. Put
\[
R_q(\rep,h)=\sqrt{|\det G_q|}, \quad q=0,...,n.
\]
Define the ``regulator'' $R(\rep,h)$ by
\[
R(\rep,h)=\prod_{q=0}^n R_q(\rep,h)^{(-1)^q}.
\]
\begin{prop}\label{reidem1}
 Let $\rep$ be a unimodular representation of $\Gamma$ on a finite-dimensional
$\R$-vector space $E$. Let $M\subset E$ be a $\Gamma$-invariant lattice and let
$\cM$ be the associated local system of finitely generated free $\Z$-modules 
on $X$. Let $h$ be
a fibre metric in the flat vector bundle $E=\cM\otimes\R$. Then we have
\[
\tau_X(\rep,h)=R(\rep,h)\cdot \prod_{q=0}^n 
\left|H^q(X,\cM)_{\tors}\right|^{(-1)^{q+1}}.
\]
\end{prop}
We will also consider complex representations. A complex representation has an
underlying real representation and we need to understand the relation 
between the Reidemeister torsion for the two representations. This is answered 
by the following lemma.
\begin{lem}\label{realcompl}
Let $\chi\colon\Gamma\to \GL(V)$ be a unimodular, acyclic  representation in a 
finite-dimensional complex vector space $V$. Let $\chi^\R$ be the 
corresponding real representation in $V^\R$. Then we have
\[
\tau_X(\rho^\R)=\tau_X(\rho)^2.
\]
\end{lem}
\begin{proof}
The $p$-simplexes of $K$ form a basis of $C^p(\tilde K;\Z)$ as a $\Z[\Gamma]$
module. With respect to these bases, the 
coboundary operator $\tilde d_p\colon C^p(\tilde K;\Z)\to C^{p+1}(\tilde K;\Z)$
is given by a matrix $(a_{ij})$ with $a_{ij}\in\Z[\Gamma]$. Let $\chi\colon
\Z[\Gamma]\to \End(V)$ be defined by 
\[
\chi\left(\sum_j n_j\gamma_j\right)=\sum_j n_j\chi(\gamma_j),\quad n_j\in\Z,\;
\gamma_j\in\Gamma.
\]
It follows that the coboundary operator $d_p^\chi \colon C^p(K;V)\to 
C^{p+1}(K;V)$ is given by the matrix $(\chi(a_{ij})$. Choose an inner product 
in $V$. Let  
\[
\Delta_p^\chi=(d_p^\chi)^*\circ d_p^\chi+d_{p-1}^\chi\circ(d_{p-1}^\chi)^*.
\]
Since $\chi$ is acyclic, $\Delta_p^\chi$ is invertible.   
By Lemma \ref{laplace1} the Reidemeister torsion is given by
\[
\tau_X(\chi)=\prod_{j=1}^n\left(\det \Delta_j^\chi\right)^{(-1)^{j+1}j/2}.
\]
A similar formula holds for $\tau_X(\chi^\R)$. Thus in order to prove the lemma
it suffices to prove that $\det \Delta_j^{\chi^\R}=(\det\Delta_j^\chi)^2$,
 $j=1,...,n$. For a linear operator $A\colon W\to W$ in a complex vector space
$W$ let $A^\R$ denote the corresponding real operator in $W^\R$. By construction
we have $\Delta_j^{\chi^\R}=(\Delta_j^\chi)^\R$. Thus it suffices to prove that
for any linear operator $A$ in a complex vector space we have $\det A^\R=
|\det A|^2$. Using the Jordan normal form of $A$, the problem is reduced to
the one-dimensional case. Let $A\colon \C\to\C$ be the multiplication by
$z=u+iv$. Then $A^\R=\begin{pmatrix}u&-v\\v&u\end{pmatrix}$ and therefore
$\det A^\R=u^2+v^2=|z|^2=|\det A|^2$ which proves the lemma.
\end{proof}

\section{Arithmetic groups}\label{arithmgr}

In this section we recall some facts about quaternion algebras and 
arithmetic subgroups of $\SL(2,\C)$ derived from quaternion algebras. 

Let $F$ be a field of characteristic zero. 
Every quaternion algebra
over $F$ can be described as $4$-dimensional $F$-vector space with basis
$1,i,j,k$ satisfying the following relations:
\[
i^2=a,\quad j^2=b,\quad ij=k,\quad ji=-k,
\]
where $a,b\in F^\times$. This algebra will be denoted
by $H(a,b;F)$ or simply $H(a,b)$. Let $x\in H\mapsto\overline x\in H$ 
be the involution defined by
\[
\overline{x_0+x_1i+x_2j+x_3k}=x_0-x_1i-x_2j-x_3k.
\]
Then the norm  and the trace in $H$ are defined by
\[
N(x)=x\overline x,\quad \Tr(x)=x+\overline x,\quad x\in H.
\]

It follows from Wedderburn's structure theorem for simple algebras that
a quaternion algebra $H$ over $F$ is either a division algebra or is 
isomorphic
to $M_2(F)$ (see \cite[Theorem II.2.7]{Lam}, \cite[Theorem 2.1.7]{MR}). In
the latter case $H$ is called split over $F$. 
 
Let $D$ be a division algebra over $F$ of degree $d^2$, $d\le 2$. 
Let $m=2/d$. There is an associated almost simple
semisimple algebraic group $\SL_m(D)$. It is defined as the functor from 
$F$-algebras to groups which assigns to an $F$-algebra $R$ the group
\[
\SL_m(D)(R):=\{x\in M_m(D\otimes_F R)\colon N(x)=1\}.
\] 
The algebraic groups $\SL_m(D)$ are the forms of $\SL_2$ over $F$
\cite[Theorem 27.9]{Min}. If $d=1$, we have $D=F$ and
$\SL_2(D)=\SL_2$. If $d=2$, $D$ is a quaternion division algebra. 

We consider $\SL_2$ as algebraic group over $F$. 
For $n\in\N$ let $V(n):=S^n(F^2)$ be the $n$-th symmetric power of $F^2$ and let
$\Sym^n\colon\SL_2\to \GL(V(n))$ be the $n$-th symmetric power of the standard
representation $\rho\colon \SL_2\to \GL_2$ of $\SL_2$. The following lemma is a 
special case of \cite[Theorem 3.3]{Tit}. For the convenience of the reader we
include the proof which  was kindly communicated to us by Skip Garibaldi.
\begin{lem}\label{twistform2}
Let $G^\prime$ be a form of $\SL_2/F$. For every even $n$  
there exists a
$F$-rational representation of $G^\prime$ on $V(n)$ which 
is equivalent to $\Sym^n$ over $\overline F$. 
\end{lem}
\begin{proof}
First we observe that the Dynkin diagram $A_1$ has no nontrivial automorphisms.
Therefore $\SL_2$ has no outer automorphims and 
\begin{equation}\label{innauto}
\Aut(\SL_2)=\PGL_2.
\end{equation}
Hence the  forms of $\SL_2$ are classified by the Galois cohomology set 
$H^1(\overline F/F,\PGL_2)$ (see \cite[pp. 181]{Min}). Let 
$\mu\colon \SL_2(\ol F)\to G^\prime(\ol F)$ be an isomorphism 
(which exists by assumption). For $\sigma\in \Gal(\overline F/F)$ put 
\begin{equation}\label{innauto1}
\varphi_\sigma=\mu^{-1}\circ\sigma\circ\mu\circ\sigma^{-1}.
\end{equation}
Then $\varphi_\sigma$ is an automorphism of $\SL_2/\overline{F}$. By 
\eqref{innauto} there exists $a_\sigma\in\PGL_2(\overline F)$ such that
$\varphi_\sigma=\Int(a_\sigma)$. As in \cite[Theorem 26.10]{Min} it follows that
$\sigma\in\Gal(\overline F/F)\mapsto a_\sigma\in\PGL_2(\overline F)$ is the 
cocycle that represents  $G^\prime$.  
Let $a^\prime_\sigma\in SL_2(\overline F)$ be any lift 
of $a_\sigma$. Let $n\in \N$ be even. Then $\Sym^n$ is trivial on
the center of $\SL_2$ and therefore the map
\[
\sigma\in\Gal(\overline{F}/F)\mapsto\Sym^n(a^\prime_\sigma)\in 
\GL(V(n)\otimes\overline{F})
\] 
is a well-defined cocycle. By Hilbert's theorem 90 we have 
$H^1(\overline F/F,\GL_N)=1$ for every $N$. Hence there exists 
$x\in \GL(V(n)\otimes \overline F)$ such that 
\begin{equation}\label{cocycle}
\Sym^n(a^\prime_\sigma)=x^{-1}\sigma(x),\quad \forall\sigma
\in\Gal(\overline F/F).
\end{equation}
Define the map $\rho_n\colon G^\prime\to \GL(V(n))$ by
\begin{equation}\label{twistrep}
\rho_n:=\Int(x)\circ \Sym^n\circ \mu.
\end{equation}
This is a representation of $G^\prime$ over $\ol F$ which over $\ol F$ is 
equivalent to $\Sym^n$.  By \eqref{innauto1} the action of 
$\Gal(\ol F/F)$ on $G^\prime(\ol F)$ is given by
\begin{equation}\label{galact}
\mu(\sigma(g))=\Int(a_\sigma)\cdot \sigma(\mu(g)),\quad 
g\in G^\prime(\ol F),\:\sigma\in\Gal(\ol F/F).
\end{equation}
Using \eqref{cocycle}, \eqref{twistrep} and \eqref{galact}, it follows that for
all $\sigma\in\Gal(\ol F/F)$ and $g\in G^\prime(\ol F)$ we get
\[
\begin{split}
\rho_n(\sigma(g))= &\Int(x)\Sym^n\left(\Int(a^\prime_\sigma)\sigma(\mu(g))\right)
=\Int(x\Sym^n(a^\prime_\sigma))\sigma(\Sym^n\left(\mu(g)\right))\\
&=\Int(\sigma(x))\sigma\left(\Sym^n(\mu(g))\right)
=\sigma(\Int(x)\Sym^n(\mu(g)))=\sigma(\rho_n(g)).
\end{split}
\]
Thus $\rho_n$ commutes with the action of $\Gal(\ol F/F)$ and hence is defined
over $F$. 
\end{proof}

Now let $F$ be an imaginary quadratic number field and let $\cO_F$ be the ring
of integers of $F$. Fix an embedding $F\subset \C$. As explained above, every
quaternion algebra over $F$, which is not isomorphic to $M_2(F)$, is a division
algebra. This is the case which is of interest for us. So let $D$ be a 
quaternion division algebra over $F$. Put
\[
D^1=\{x\in D\colon N(x)=1\},\quad D^0=\{x\in D\colon \Tr(x)=0\}.
\]
Let 
\[
G=R_{F/\Q}\SL_1(D)
\] 
be the algebraic group obtained from $\SL_1(D)$ by restriction of scalars. 
Then $G$ is defined over $\Q$. We have
\begin{equation}\label{restrscal1}
G(\Q)\cong D^1,\quad G(\R)\cong\SL_1(D)(F\otimes_\Q\R)=(D\otimes_F\C)^1
\end{equation}
and there is an isomorphism of $\C$-algebras
\begin{equation}\label{isodiv}
\varphi\colon (D\otimes_F\C)^1\stackrel{\cong}\longrightarrow\SL(2,\C).
\end{equation}
Let $\ho$ be an order in $D$. Recall that this means that $\ho$ is a finitely 
generated $\cO_F$-module which
contains an $F$-basis of $D$ and which is also a subring of $D$.
Let 
\[
\ho^1=\{x\in\ho\colon N(x)=1\},\quad \ho^0=\{x\in\ho\colon \Tr(x)=0\}.
\]
Then $\ho^1\subset D^1$ and $\ho^0\subset D^0$ is a lattice which is invariant
under $\ho^1$ with respect to the adjoint action of $\SL_1(D)$ on $D^0$.
Thus by \eqref{restrscal1} it follows that $\ho^1$ corresponds to an arithmetic 
subgroup $\Gamma^1\subset G(\Q)$. Put
\[
\Gamma=\varphi(\ho^1)
\]
Then $\Gamma$ is a discrete subgroup of $\SL(2,\C)$. Such a group is called
an arithmetic subgroup derived from a quaternion algebra.  
The following lemma is a consequence of \cite[Theorem 10.1.2]{EGM}. 
\begin{lem}\label{cocompact}
Let $\Gamma\subset \SL(2,\C)$ be derived from a quaternion division algebra over
$F$. Then the quotient $\Gamma\bs\SL(2,\C)$ is compact. 
\end{lem}
By  imposing further conditions on $D$, one can achieve that $\Gamma$ is 
torsion free \cite[Theorem 10.1.2]{EGM}. Alternatively, we can pass to a 
normal subgroup of $\Gamma$ of finite index which is torsion free. 

Let $V_1(n)=R_{F/\Q}(V(n))$. Then $V_1(n)$ is a $\Q$-vector space and it follows
from Lemma \ref{twistform2} that for every even $n$ there exist a $\Q$-rational
representation
\[
\rho_n\colon G\to\GL(V_1(n))
\]
which over $\R$ is equivalent to $\Sym^n$ of $\SL_2$. Since $\Gamma^1$ is an
arithmetic subgroup of $G(\Q)$, there exists a lattice $L\subset V_1(n)$ which
is stable under $\Gamma^1$ (see \cite[Proposition 28.9]{Min}). Using the 
isomorphism \eqref{isodiv} this implies the following proposition.
\begin{prop}\label{lattice}
Let $\Gamma\subset\SL(2,\C)$
be an arithmetic subgroup derived from a quaternion  algebra over an imaginary
quadratic number field.
Then  for each even $n\in\N$ there exists a lattice
$M_n\subset S^n(\C^2)$ which is stable under $\Gamma$ with respect to $\Sym^n$.
\end{prop}
We add some remarks about the classification of quaternion algebras over a
number field $F$. For details see \cite[Chap. III]{Vig}. Let $H$ be a 
quaternion algebra over $F$. Given a place $v$ of $F$, let
\[
H_{v}=H\otimes_F F_{v}.
\]
Over $\C$ every quaternion algebra splits, i.e., there
is an $\C$-algebra isomorphism 
\[
\varphi\colon H\otimes_F\C\stackrel{\cong}\longrightarrow M_2(\C).
\]
By the Frobenius theorem, a quaternion algebra over $\R$ is either split or
isomorphic to the Hamiltonian quaternions. If $\pg$ is a finite place of $F$,
then, up to isomorphism, there is a unique quaternion division algebra  
over $F_{\pg}$ \cite[Theorem II.1.1]{Vig}. $H$ is called ramified at a given
place $v$, if $H_v$ is a division algebra. If $H=H(a,b;F)$, $a,b\in F^\times$, 
the behavior at $v$ is determined by the Hilbert symbol $(a,b)_v\in\{\pm1\}$. 
$H_v$ is split iff 
$(a,b)_v=1$. It follows from the Hasse-Minkowski principle that 
$H$ splits over $F$ iff $H_v$ splits for every place $v$ of $F$
\cite[Theorem 2.9.6]{MR}. For $H=H(a,b;F)$ this is equivalent to $(a,b)_v=1$ 
for all places $v$. Furthermore the number of places where $H$ is ramified
is even. Denote by $V_f(F)$ and $V_\R(F)$ the set of finite places and real
places of $F$, respectively. Then for every finite set $S\subset V_f(F)\cup 
V_\R(F)$ of even cardinality there is a quaternion algebra $H$ over $F$ 
which is unique up to $F$-algebra isomorphism, such that the set of places
where $H$ is ramified is equal to $S$. 

Thus for an imaginary quadratic number field $F$, we can pick any nonempty
finite set $S$ of finite  places  $\pg$ with $|S|$ even. Then up to 
isomorphism, there is a unique quaternion division algebra $D$ over $F$
which is ramified exactly at the places $\pg\in S$.

\section{Bounds for torsion in cohomology}\label{bounds}

We now prove \eqref{estim1} of theorem \ref{thm-asymp}, which states that the 
contribution of the terms with cohomological degree 1 and 3 to the alternating 
sum  \eqref{asymp2} is small.  To begin with the $H^3$ term, it is required to 
show that

\begin{equation}
\label{slow}
\ln |H^3( X, \cM_{2k} )| \ll k \ln k
\end{equation}
uniformly over all choices of lattice $M_{2k}$.  We first apply the 
isomorphism $H^3( X, \cM_{2k} ) \simeq (M_{2k})_\Gamma$, which follows from 
computing $H^3$ using a triangulation of $X$ and then observing that this is 
equivalent to computing $H_0(\Gamma, M_{2k})$ using the dual triangulation.  
We shall then bound $(M_{2k})_\Gamma$ by working locally.  Let 
$V(2k)=S^{2k}(F^2)$.
For each prime $\pg$ of $F$, let $V_\pg(2k)$ be the completion of $V(2k)$ at 
$\pg$, and let $M_{2k, \pg}$ be the completion of the image of $M_{2k}$.  
$M_{2k, \pg}$ is a $\Gamma$-stable lattice, and we have

\begin{equation*}
\ln | (M_{2k})_\Gamma | = \sum_\pg \ln | (M_{2k, \pg})_\Gamma |.
\end{equation*}
We divide the primes of $F$ into two sets, which we call 
unramified and ramified.  The first set contains those at which the division 
algebra $D$ is unramified and the closure of $\Gamma$ in 
$D^1_\pg \simeq \SL_2(F_\pg)$ is isomorphic to the standard maximal compact, 
and the second contains the remainder.  The following lemma implies that the 
unramified primes whose residue characteristic is greater than $2k$ make no 
contribution to the sum.

\begin{lem}
\label{largep}
If $\F_q$ is the field with $q = p^j$ elements, then the $d$-th symmetric 
power representation 
of $\SL_2(\F_q)$ over $\F_q$ is irreducible for $d < p$. 
\end{lem}

\begin{proof}
Denote this representation by $(\rho, V_d)$, and let $N$ and $\overline{N}$ be 
nontrivial upper and lower unipotent elements of $\SL_2(\F_q)$.  When 
$\rho(N) - I$ and $\rho(\overline{N}) - I$ are expressed in the standard 
monomial basis of $V_d$ they are strictly upper (resp. lower) triangular, and 
all their entries in the spaces lying immediately above (resp. below) the 
diagonal are nonzero.  It follows that the only subspaces of $V_d$ which are 
invariant under both of these operators are $\{ 0 \}$ and $V_d$.
\end{proof}

Because $F$ was 
imaginary quadratic, lemma \ref{largep} allows us to assume that 
$N\pg \le 4k^2$.  The following proposition gives a 
bound for $| (M_{2k, \pg})_\Gamma |$ at unramified primes which is uniform in 
$\pg$ and $M_{2k}$, and gives us the required bound (\ref{slow}) when summed 
over those $\pg$ with $N\pg \le 4k^2$.  For ramified primes we shall adapt the 
proof of the 
proposition to give bounds which are only uniform in the lattice, but this 
will be sufficient for our purposes as the number of such primes is bounded 
independently of $k$.

\begin{prop}
\label{local}
Let $K$ be a $p$-adic field, $\cO \subset K$ the ring of integers in $K$, 
$\varpi$ a uniformiser of $\cO$, and $q = | \cO / \varpi |$.  Consider the 
$2k$-th symmetric power representation of $G = \SL(2,\cO)$ on a vector space 
$V$ of dimension $2k+1$.  If $L \subset V$ is any $G$-stable lattice and 
$L/L'$ is the largest $G$-invariant quotient of $L$, then we have
\begin{equation*}
\ln_q | L : L' | \ll k/q
\end{equation*}
where the implied constant is absolute, i.e. independent of $L$ and $K$.
\end{prop}

\begin{proof}
The proof proceeds in two steps.  First, choose $a \in \cO^\times$ to have the 
highest possible order in all the quotient groups 
$\cO^\times / ( 1 + \varpi^t \cO)$, or equivalently so that $a$ is a primitive 
root in $(\cO/\varpi)^\times$ and $a^{q-1} \notin 1+ \varpi^2 \cO$.  Define $T$ 
to be the diagonal matrix
\begin{equation*}
T = \left( \begin{array}{cc} a & 0 \\ 0 & a^{-1} \end{array} \right).
\end{equation*} 
We know that $(T-1)L \subset L'$, so it would suffice to bound the order of 
$L / (T-1)L$.  This is infinite, however, because $T$ fixes the vector 
$(xy)^k$.  Instead, we shall prove that if $V_0$ is the space spanned by 
$(xy)^k$ and we define $J$ and $J'$ to be the intersections of $V_0$ with $L$ 
and $L'$, we have
\begin{equation}
\label{Linduct}
\ln_q | L : L' | \ll \frac{k}{q} + \ln_q | J : J' |.
\end{equation}
The second step is to bound $|J : J'|$, for which we shall need to consider 
the action of the unipotents on $V$.

We carry out the first step by decomposing $V$ into a direct sum of subspaces 
$\cV_i$, for $i \in \Z_{\ge 0}$ and $i = \infty$.  For $i \in \Z_{\ge 0}$, we 
define $\cV_i$ to be the span of the eigenspaces of $T-1$ with eigenvalue 
$\lambda$ satisfying $\ord_\varpi \lambda = i$, and define 
$\cV_\infty = \text{span} \{ (xy)^k \}$.  In other words, if $V_t$ is the span 
of the monomial $x^{k+t} y^{k-t}$, we define

\begin{eqnarray*}
\cV_0 & = & \bigoplus_{ 2t \not\equiv 0\; (q-1) } V_t \\
\cV_i & = & \bigoplus_{ \substack{ 2t = j \; (q-1) p^{i-1} \\ 
j \not\equiv 0 \; (p) } } V_t \\
\cV_\infty & = & V_0.
\end{eqnarray*}
Note that for $i \ge 1$ we have the bound

\begin{equation*}
\dim \cV_i \le Ck / q p^{i-1}
\end{equation*}
for some absolute constant $C$.  Let $\Pi_i$ be the projection onto $\cV_i$, 
and define $\cV_i'$ by

\begin{equation*}
\cV'_i = \bigoplus_{ \substack{j > i, \\ j = \infty} } \cV_j
\end{equation*}
so that $\cV'_i$ is the span of the eigenspaces of $T-1$ with eigenvalue 
$\lambda$ satisfying $\ord_\pg \lambda > i$.  Let $L_i = L \cap \cV_i'$ and 
$L'_i = L' \cap \cV_i'$.  We shall prove (\ref{Linduct}) using induction on 
$L_i$, by establishing the inequalities

\begin{eqnarray}
\label{Linduct1}
| L : L' | & \le & | L_0 : L_0' |,\\
\label{Linduct2}
\ln_q | L_i : L_i' | & \le & Ck(i+1)/qp^i + \ln_q | L_{i+1} : L_{i+1}' |
\end{eqnarray}
for all $i \ge 0$.

To prove (\ref{Linduct1}), note that because $T$ commutes with all co-ordinate projections we have

\begin{equation*}
(T - 1) \Pi_0 L \subseteq \Pi_0 L' \subseteq \Pi_0 L,
\end{equation*}
and that all the above inclusions must be equalities because the determinant 
of $T-1$ restricted to $\cV_0$ is a unit of $\cO$.  To deduce (\ref{Linduct1}) 
from the equality $\Pi_0 L = \Pi_0 L'$, let $\{ v_i \}$ be a system of coset 
representatives for $L_0 / L_0'$, and let $x \in L$ be given.  Because 
$\Pi_0 L = \Pi_0 L'$, there exists $y \in L'$ such that $\Pi_0 y = \Pi_0 x$, 
and so $x-y \in L_0$.  Therefore there exists $v_i$ such that 
$x - y - v_i \in L_0'$, which implies $x \in L' + v_i$.

Applying the same argument to the lattices $L_i$ and $L_i'$ gives

\begin{equation*}
(T-1) \Pi_{i+1} L_i \subseteq \Pi_{i+1} L'_i \subseteq \Pi_{i+1} L_i,
\end{equation*}
and on taking determinants we have

\begin{equation*}
\ln_q | \Pi_{i+1} L_i : \Pi_{i+1} L'_i | \le \ln_q | \Pi_{i+1} L_i : 
(T-1)\Pi_{i+1} L_i | \le (i+1) {\dim \cV_{i+1}} \le Ck(i+1)/qp^i,
\end{equation*}
where $i+1$ appears because it is the $\varpi$-adic valuation of the 
eigenvalues of $T-1$ on $\cV_{i+1}$.  Repeating the above argument on coset 
representatives, we see that this implies (\ref{Linduct2}).
On summing (\ref{Linduct1}) and (\ref{Linduct2}) we obtain

\begin{eqnarray*}
\ln_q | L : L' | & \ll & \sum_{t=1}^i \frac{kt}{q p^{t-1} } 
+ \ln_q | L_i : L_i' | \\
& \ll & \frac{k}{q} + \ln_q | L_i : L_i' |,
\end{eqnarray*}
from which (\ref{Linduct}) follows by letting $i \rightarrow \infty$.\\

To bound $|J:J'|$, we may assume that $J$ is generated by $(xy)^k$, and act on 
this monomial by an upper triangular element $N \in \SL(2,\cO)$.  $(N-1)(xy)^k$ 
will then contain a nonzero term of the form $x^{k+1} y^{k-1}$, and using the 
endomorphism algebra generated by $T$ we may in fact show that $L'$ contains a 
monomial $\varpi^s x^{k+1} y^{k-1}$ with $s$ small.  Repeating this argument 
with a lower triangular matrix $\overline{N}$ gives the required bound on 
$|J:J'|$.  The statement we require about the endomorphism algebra generated 
by $T$ is the following:

\begin{lem}
For $\gamma \in \cO^\times$, let $E$ be the ring of endomorphisms of 
$\cO^{2k+1}$ generated over $\cO$ by the diagonal matrix $A$ with entries 
$1, \gamma, \ldots, \gamma^{2k}$.  Then for $0 \le j \le 2k$, $E$ contains 
$p_j(\gamma^j) \pi_j$ where $\pi_j$ is the projection onto the $j$th 
co-ordinate and $p_j(z)$ is the polynomial

\begin{equation*}
p_j(z) = \prod_{ \substack{ 0 \le i \le 2k, \\ i \neq j } } (z - \gamma^i).
\end{equation*}

\end{lem}

\begin{proof}

This follows easily by considering the matrix $p_j(A)$, whose only nonzero 
entry is $p_j(\gamma^j)$ in the $(j,j)$th co-ordinate.

\end{proof}

Consider the polynomial $x^k (y+x)^k - (xy)^k \in L'$.  Applying the lemma 
with the choice of $\gamma = a^2$ and the projection onto the space spanned 
by $x^{k+1} y^{k-1}$, we see that $L'$ contains
\begin{equation*}
k \prod_{ \substack{ -2k-2 \le 2i \le 2k-2,\\ i \neq 0 } } (1 - a^{2i} ) x^{k+1} 
y^{k-1} = u \varpi^\alpha x^{k+1} y^{k-1},
\end{equation*}
where $u$ is a unit of $\cO$ and $\alpha \in \Z$.  In the product above, at 
most $Ck/q$ terms are divisible by $\varpi$, at most $Ck/qp$ are divisible 
by $\varpi^2$ and so on, so that we may bound $\alpha$ by
\begin{eqnarray*}
\alpha & \ll & \ln_q k + \frac{k}{q} + \frac{k}{qp} + \ldots \\
& \ll & \frac{k}{q}.
\end{eqnarray*}
By applying the same argument to the monomial $\varpi^\alpha x^{k+1} y^{k-1}$, 
but now with an element of the opposite unipotent,  we see that 
$\varpi^\beta (xy)^k \in L'$ with $\beta \ll k/q$.  This gives the required 
bound on $|J : J'|$, and completes the proof.
\end{proof}

Using proposition \ref{local}, we see that the contribution to the order of 
$\log | (M_{2k})_\Gamma |$ from unramified primes is at most a constant times
\begin{equation*}
\sum_{N\pg \le 4k^2} \frac{k \ln N\pg}{ N\pg} \ll k \ln k.
\end{equation*}
As the number of ramified primes is bounded independently of $k$, in the 
remaining cases it suffices to prove bounds of the form 
$\ln |(M_{2k,\pg})_\Gamma| \ll k$ where the constant is allowed to depend on 
$\pg$.

First, consider a prime $\pg$ at which $D$ is split but the closure of 
$\Gamma$ in $D_\pg^1 \simeq \SL_2(F_\pg)$ is not isomorphic to  $\SL(2,\cO)$, 
and denote this closure by $G$.  As in the unramified case, let $V$ be the 
$2k$-th degree symmetric power representation of $\SL_2(F_\pg)$, let 
$L \subset V$ be any $G$-stable lattice, and $L/L'$ its largest $G$-invariant 
quotient.  As $G$ is open, we know it contains an element

\begin{equation*}
T = \left( \begin{array}{cc} a & 0 \\ 0 & a^{-1} \end{array} \right)
\end{equation*} 
for some $a \in \cO^\times$.  As in the unramified case, define $\cV_i$ to be 
the sum of the eigenspaces of $T-1$ with eigenvalue $\lambda$ satisfying 
$\ord_\varpi \lambda = i$, let $\cV_i'$ be the sum of the eigenspaces with 
$\ord_\varpi \lambda > i$, and let $J = L \cap \cV_\infty$ and 
$J' = L' \cap \cV_\infty$.  Becuase we have inequalities

\begin{equation*}
\dim \cV_i \le C(a) \frac{k}{p^i}
\end{equation*}
uniformly in $i$ and $k$, we may use these subspaces to perform the same 
reduction argument which bounds $|L : L'|$ in terms of $|J : J'|$ to obtain

\begin{equation*}
\log_q |L : L'| \le \log_q |J : J'| + O_a(k).
\end{equation*}

$G$ also contains upper and lower unipotent elements

\begin{equation*}
N = \left( \begin{array}{cc} 1 & n \\ 0 & 1 \end{array} \right) \quad 
\text{and} \quad \overline{N} = \left( \begin{array}{cc} 1 & 0 \\ n & 1 
\end{array} \right)
\end{equation*} 
for some $n \in \cO$, and arguing as in the unramified case we may use these 
to show that

\begin{eqnarray*}
\log_q |J : J'| & \ll_n &  \log_q k + \sum_{ \substack{ -2k-2 \le 2i \le 2k+2,\\ 
i \neq 0 } } \ord_\varpi (1 - a^{2i} ) \\
& \ll_{n,a} & \log_q k + \sum_{i \ge 0} \frac{k}{p^i} \\
& \ll_{n,a} & k.
\end{eqnarray*}
We therefore have an upper bound of $\log_q |L : L'| \ll k$, where the 
constant may depend on $\pg$.

Finally, suppose that $D$ is ramified at $\pg$.  We continute to denote the 
closure of $\Gamma$ in $D_\pg^1$ by $G$.  Using the adjoint representation, we 
may realise $D^1_\pg$ as an algebraic subgroup of $GL_3(F_\pg)$.  There is an 
open subgroup $U$ of $GL_3(F_\pg)$ such that the power series $\ln(1 + A)$ 
converges $p$-adically for $1 + A \in U$, from which we see that there is an 
open subgroup $U' \subset D_\pg^1$ on which we may define an inverse 
exponential map.  By applying this map to $U' \cap G$ we obtain a lattice $N$ 
in the Lie algebra of $D_\pg^1$ which annihilates $L/L'$.  We may then extend 
scalars to a field over which $D_\pg$ splits so that we are again dealing with 
$\SL_2$, and apply the method of proposition \ref{local} to obtain an upper 
bound of $\log_q |L:L'| \ll k$.  This completes the proof of (\ref{slow}).\\

In the case of $H^1$, we may apply the long exact sequence in cohomology 
associated to the short exact sequence

\begin{equation*}
0 \longrightarrow M_{2k} \overset{ \times d}{\longrightarrow} M_{2k} 
\longrightarrow M_{2k} / d \longrightarrow 0,
\end{equation*}
of which one segment is
\begin{equation*}
H^0( \Gamma, M_{2k} / d) \longrightarrow H^1( \Gamma, M_{2k}) 
\overset{ \times d}{\longrightarrow} H^1( \Gamma, M_{2k}).
\end{equation*}
Therefore the torsion in $H^1( \Gamma, M_{2k})$ is bounded by the limit of 
the $\Gamma$-invariants in $M_{2k} / d$ as $d$ becomes divisible by 
arbitrarily high powers of every prime.  We may again bound this by working 
locally, and so wish to bound the order of the module of $\Gamma$-invariants 
$(M_{2k,\pg} / \varpi^t M_{2k,\pg})^\Gamma$ for each $\pg$ and arbitrary $t$.  
If we let $L$ be the inverse image of 
$(M_{2k,\pg} / \varpi^t M_{2k,\pg})^\Gamma$ in $V_\pg(2k)$, we see that 
$\Gamma$ acts trivially on $L / \varpi^t M_{2k,\pg}$.  We may therefore apply 
our bounds for $L_\Gamma$ to obtain

\begin{equation*}
\ln | H^1( \Gamma, M_{2k}) | \ll k \ln k
\end{equation*}

as required.

\section{Proof of the main results}\label{main}

Let $X=\Gamma\bs \bH^3$ be a compact oriented hyperbolic $3$-manifold defined
by a cocompact torsion free discrete subgroup $\Gamma\subset\SL(2,\C)$.  For
$m\in\N$ let $\rp_m\colon\SL(2,\C)\to \GL(S^m(\C^2))$ be the $m$-th symmetric 
power of the standard representation of $\SL(2,\C)$.   
We regard $S^m(\C^2)$ as an 
$\R$-vector space.  Let $\rp_m^\R$ be the corresponding real 
representation. Let $E_m\to X$ be the flat vector bundle associated to
$\rp_m$ or $\rp_m^\R$, respectively. By \cite[Chapt. VII, Theorem 6.7]{BW} we
have $H^*(X,E_m)=0$. Therefore, the Reidemeister torsions $\tau_X(\rp_m)$ and
$\tau_X(\rp_m^\R)$ of $X$ with respect to the
restrictions to $\Gamma$ of $\rp_m$ and $\rp_m^\R$, respectively, are well 
defined. By Lemma \ref{realcompl} we have
\begin{equation}\label{realtor}
\tau_X(\rp_m^\R)=\tau_X(\rp_m)^2.
\end{equation}
Then it follows from 
\cite[Corollary 1.2]{Mu2} that
\begin{equation}\label{reidasymp2}
-\log\tau_X(\rp_m^\R)=\frac{\vol(X)}{2\pi} m^2+O(m)
\end{equation}
as $m\to\infty$. 

Now let $\Gamma\subset\SL(2,\C)$ be an arithmetic subgroup derived from a 
quaternion devision algebra over an imaginary quadratic number field $F$.
By Lemma \ref{cocompact} it is cocompact. By passing to a normal subgroup of 
finite index we can assume that $\Gamma$ is torsion free. 
Let $m\in\N$ be even. By Proposition 
\ref{lattice} there exists a lattice $M_{m}\subset S^{m}(\C^2)$ 
which is stable under $\Gamma$.
Let $\cM_{m}$ be the associated local system of free $\Z$-modules over $X$.  
Let $\cM_{m}(\R):=\cM_{m}\otimes\R$. 
This is the local system associated to the restriction of $\rho_{m}$ to
$\Gamma$. Hence $H^*(X,\cM_m(\R))\cong H^*(X,E_m)$ and it follows from the 
remark above that $H^*(X,\cM_m(\R))=0$.  
Therefore $H^*(X,\cM_{n})$ is a finite abelian group. Note that $H^0(X,\cM_n)
=M_n^\Gamma=0$. 
Denote by $|H^p(X,\cM_n)|$ the order of $H^p(X,\cM_n)$, $p=1,2,3$. 
Then by Proposition \ref{reidem1} we get
\begin{equation}\label{reidem2}
\tau_X(\rp_n^\R)=\prod_{p=1}^3|H^p(X,\cM_n)|^{(-1)^{p+1}}.
\end{equation}
Combining \eqref{reidasymp2} and \eqref{reidem2}, we obtain Theorem 
\ref{thm-asymp}.

\noindent
{\bf Remark}: It follows from \eqref{reidem2} that the quantity 
$\sum_{p=1}^3(-1)^p\ln |H^p(X,\cM_{2k})|$ is independent of the choice of 
lattice $M_{2k} \subset V(2k)$, but this may also be deduced in an elementary 
way from the long exact sequence in cohomology associated to any inclusion of 
lattices $M'_{2k} \subset M_{2k}$.  Let $T=M_{2k}/M'_{2k}$. The long exact 
sequence associated to
\begin{equation*}
0 \longrightarrow M'_{2k} \longrightarrow M_{2k} \longrightarrow T 
\longrightarrow 0
\end{equation*}
implies that
\begin{equation}
\label{latticeinvar}
\sum_{p=1}^3(-1)^p\ln |H^p(X,\cM_{2k})| - \sum_{p=1}^3(-1)^p
\ln |H^p(X,\cM'_{2k})| = \sum_{p=1}^3(-1)^p\ln |H^p(X,\cT)|.
\end{equation}
The groups $H^p(X,\cT)$ are the cohomology groups of a complex $\{ C^p(T) \}$, 
where $C^p(T)$ is the group of $T$-valued cochains of degree $p$ in some 
triangulation of $X$.  Because $X$ is three dimensional, its Euler 
characteristic $\chi(X)$ is zero and so we have
\begin{equation*}
\sum_{p=1}^3(-1)^p\ln |C^p(T)| = \chi(X)\log|T|=0.
\end{equation*}
This implies the same relation for the groups $H^p(X,\cT)$, and so 
\eqref{latticeinvar} is zero as required.

The proof of Theorem \ref{value1} follows from \cite[Theorem 1.5]{Mu2}. Let
$\rp\colon SL(2,\C)\to \GL(V)$ be an irreducible finite-dimensional complex
representation of $\SL(2,\C)$, regarded as real Lie group. 
The Ruelle zeta function $R_\rp(s)$ considered in \cite{Mu2}
is related to the zeta function $R(s;\rp)$ by 
\[
R(s;\rp)=R_\rp(s)^{-1}.
\] 
The restriction of $\rp$ to $\Gamma$ defines a flat vector bundle 
$E_\rp\to X$. By \cite[Lemma 3.1]{MM} it carries a canonical fibre metric 
$h$. Let $H^*(X,E_\rp)$ be 
the de Rham cohomology of $E_\rp$-valued differential forms. 
Let $\theta$ be the Cartan involution of $\SL(2,\C)$ with respect to $SU(2)$. 
Let $\rp_\theta:=\tau\circ\rp$. If  $\rp_\theta\not\cong\rp$, then it
follows from \cite[Chapt. VII, Theorem 6.7]{BW} that $H^*(X,E_\rp)=0$. 
Let $T_X(\rho)$ be the Ray-Singer analytic torsion of $X$ with respect
to the restriction of $\rp$ to $\Gamma$. Note that in order to define the 
analytic torsion, we
need to choose a fibre metric in $E_\rp$. However, since $H^*(X,E_\rp)=0$ and
the dimension of $X$ is odd, the analytic torsion is independent of the any 
fibre metric \cite[Corollary 2.7]{Mu1}, which justifies the notation.
By \cite[Theorem 1.5, 1)]{Mu2} the Ruelle zeta function
$R(s;\rp)$ is regular at $s=0$ and we have
\begin{equation}\label{ruellval1}
|R(0;\rp)|=T_X(\rp)^{-2},
\end{equation}
Furthermore, by \cite[Theorem 1]{Mu1}, the analytic torsion $T_X(\rp)$ equals
the Reidemeister torsion $\tau_X(\rp)$ of $X$ and $\rp|_\Gamma$. 
Let $\rho^\R$ be the real representation associated to $\rho$.
Together with
Lemma \ref{realcompl} we obtain
\begin{equation}\label{ruellreid1}
|R(0;\rp)|=\tau_X(\rp^\R)^{-1}.
\end{equation}
Now assume that $\rp_\theta=\rp$. Then $H^*(X,E_\rp)$ may  be nonzero.
It follows from \cite[Theorem 1.5, 2)]{Mu2} that
the order of $R(s;\rp)$ at $s=0$ is given by
\begin{equation}\label{order3}
\ord_{s=0}R(s;\rp)=2\sum_{p=0}^3 (-1)^p p\dim H^p(X,E_\rp)
\end{equation}
and the leading coefficient $R^*(0;\rp)$ of the Laurent expansion of 
$R(s;\rp)$ at $s=0$ equals the analytic torsion $T_X(\rp;h)^{-2}$, where 
$h$ is the canonical fibre metric $h$ on $E_\rp$. Let $\tau_X(\rp;h)$ be the 
Reidemeister torsion with 
respect to the $L^2$-inner product in $H^*(X,E_\rp)$ defined by the 
isomorphism with the space 
$\cH^*(X,E_\rp)$ of  $E_\rp$-valued harmonic forms. Using again 
\cite[Theorem 1]{Mu1} and Lemma \ref{realcompl}, we get
\begin{equation}\label{laurent3}
R^*(0;\rp)=\tau_X(\rp^\R;h)^{-1}.
\end{equation}
Now assume that $M\subset V$ is a lattice which is stable under $\Gamma$ with
respect to $\rp$. Let
$\cM$ be the associated local system of free $\Z$-modules and let $\cM(\R)=
\cM\otimes\R$. If $\rp_\theta\not\cong\rp$, we have $H^*(X,\cM(\R))=0$. 
Combining Theorem \ref{reidem1} and \eqref{ruellreid1} we get
\[
|R(0;\rp)|=\prod_{p=0}^3|H^p(X,\cM)_{\tors}|^{(-1)^p}.
\]
Now assume that $\rp_\theta=\rp$. Then $\rk H^p(X,\cM)= 2\dim H^p(X,E_\rp)$.
Let $\rp\neq 1$. Then it follows from \eqref{order3} that
\[
\ord_{s=0}R(s;\rp)=\sum_{p=1}^3 (-1)^p\ra H^p(X,\cM).
\]
Using  Theorem \ref{reidem1} and \eqref{laurent3}, we get
\[
R^*(0;\rp)=R(\cM)^{-1}\prod_{p=0}^3|H^p(X,\cM)_{\tors}|^{(-1)^p}.
\]
Finally, if $\rp=1$ it follows from \cite[Theorem 1.5]{Mu2} that the order
of $R(s;1)$ at $s=0$ equals $2\dim H^1(X,\R)-4$.

\end{document}